\documentclass{PJM}

\authorone{Amra Reki\' c-Vukovi\' c and Nermin Oki\v ci\' c}
\addressone{Department of Mathematics, Faculty of Natural Sciences and Mathematics, University of Tuzla, Univerzitetska 4, 75000 Tuzla}
\countryone{Bosnia and Herzegovina}
\emailone{amra.rekic@untz.ba, nermin.okicic@untz.ba}
\title{ANOTHER LOOK AT THE CONTINUITY OF MODULI OF NONCOMPACT CONVEXITY}
\headlinetitle{Another look at the continuity}

\begin{document}
\centerline{Communicated  by  Ayman Badawi}
\vskip0.1in
\normalsize

 {\bf Abstract} In this article we introduce the new modulus $\triangle'_{X,\phi}(\varepsilon)$, for which we prove that in the general case is different from the classical modulus of noncompact convexity. The main result of the paper is showing the continuity of the modulus of noncompact convexity for arbitrary minimalizable (strictly minimalizable) measure of noncompactness on arbitrary metric space. \vspace*{5mm}

\section{Introduction}
There are many ways to describe geometrical properties of Banach spaces. The most common way is by defining the real function, so called modulus, which depends on the Banach space that we consider. Usually, with modulus we define a proper constant or coefficient that is directly related to the modulus. The value of the coefficient tells us more about the properties of the space. The classical modulus of convexity, that was introduced by Clarkson \cite{clarkson1936uniformly}, that defines uniformly convex spaces is the origin for many other moduli that were introduced later. Similarly, the property of uniform smoothness of Banach spaces was defined using the Lindenstrauss modulus of smoothness \cite{lindenstrauss}. Prus described the uniform Opial property of Banach spaces by using Opial modulus \cite{Prus}. Property of near uniform convexity of Banach spaces was defined by the modulus of noncompact convexity, more precisely by Goebel-Sekowski, Banas and Dominguez-Lopez moduli \cite{Ayerbe_knjiga}. Analysis of the properties of the moduli and their characteristics additionally contributes to understanding geometrical properties of the Banach spaces. In this way we get the classification of Banach spaces and better connection with the theory of fixed point. \\
We know some results for some properties of the modulus $\triangle _{X,\phi}(\varepsilon)$ for an arbitrary (strictly) minimalizable measure of noncompactness $\phi$ and Banach space $X$ with Radon-Nikodym property, as well as the result for continuity of the modulus $\triangle _{X,\phi}(\varepsilon)$ \cite{AIN}, \cite{ANIV}. The result of the continuity was a consequence of the result that Prus gave connecting continuity of the modulus $\triangle _{X,\chi}(\varepsilon)$ to the uniform Opial condition which implies normal structure of the space \cite{Prus}. \\
In this paper, using the notion of the $\phi$-minimal set, we define a new function, i.e. the new modulus  $\triangle'_{X,\phi}(\varepsilon)$. Using properties of the new modulus, we prove continuity of the modulus of noncompact convexity $\triangle _{X,\phi}(\varepsilon)$ on $[0,\phi (B_X))$, for arbitrary minimalizable (strictly minimalizable) measure of noncompactness $\phi$ and arbitrary metric space $X$.
\subsection{Fundamental concepts and definitions}
In this paper $X$ denotes metric space, $B(x, r)$ an open ball centered at $x$ of radius $r$ and $B_X$ the unit ball in X. If $A\subseteq X$ we denote by $\overline{A}$ the closure of a set $A$ and by $coA$ the convex hull of $A$.
\begin{definition}
Let $\mathcal{B}$ be a family of bounded subsets of $X$. We call the mapping $\phi:\mathcal{B}\rightarrow [0,+\infty)$ the measure of noncompactness defined on $X$ if it satisfies the following :
\begin{enumerate}
\item $\phi (B)=0$ if and only if $B$ is relatively compact set,
\item $\phi (B)=\phi (\overline{B})$, for all $B\in \mathcal{B}$,
\item $\phi (B_1\cup B_2)\leq \max (\phi (B_1), \phi (B_2))$, for all $B_1,B_2\in \mathcal{B}$.
\end{enumerate}
\end{definition}
\noindent Some of well known measures of noncompactness are Kuratowski measure $\alpha$,
$$\alpha (A)=\inf \{\varepsilon >0 \ | \ A \textrm{ can be covered by finitely many sets of diameter } \ \leq \varepsilon\} \ ,$$
Hausdorff measure $\chi$,
$$\chi (A)=\inf \{\varepsilon >0 \ | \ A \textrm{ can be covered by finitely many balls of radius } \ \leq \varepsilon\} \ ,$$
and Istratescu measure $\beta$,
$$\beta (A)=\sup \{r>0 \ | \ A \textrm{ has an infinite $r$-separation}\} \ .$$
For more details on this measures see e.g. \cite{Akhmerov}, \cite{Ayerbe_knjiga}.\\
The notion of a $\phi$-minimal set for the measure of noncompactness $\phi$ was introduced by Benavides \cite{benavides1986some}, while studying the relation between condensing operators for the Hausdorff and Kuratowski measure of noncompactness.

\begin{definition}
 Let $X$ be a metric space and let $\mathcal{B}$ be a family of all bounded subsets of $X$. The infinite set $A\in \mathcal{B}$ is called $\phi$-minimal if  $\phi(A)=\phi(B)$ for every infinite set $B\subset A$.
\end{definition}

\noindent We call a measure of noncompactness $\phi$ a minimalizable measure of noncompactness if for every infinite bounded set $A$ and for every $\varepsilon >0 $ there exists a subset $B\subset A$ which is $\phi$-minimal and such that $\phi (B)\geq \phi (A)-\varepsilon$. A measure $\phi$ is a strictly minimalizable measure of noncompactness if for every infinite, bounded set $A$ there exists a subset $B\subset A$, which is $\phi$-minimal and such that $\phi (B)=\phi (A)$. Clearly, every strictly minimalizable measure is a minimalizable measure of noncompactness as well. See e.g. \cite{Akhmerov} and \cite{Ayerbe_knjiga} for more on minimalizable measures of noncompactness.

\begin{definition}
A modulus of noncompact convexity associated with an arbitrary measure of noncompactness $\phi$ is a function $\triangle _{X,\phi}: [0,\phi (B_X)]\rightarrow [0,1]$ given by
\begin{equation}{\label{definicijamodula}}
\triangle _{X,\phi}(\varepsilon)=\inf \{1-d(0,co(A)) \ | \ A\subseteq \overline{B}_X,  \ \phi(A)>\varepsilon \}.
\end{equation}
\end{definition}

\noindent Banas considered a modulus $\triangle _{X,\phi}(\varepsilon)$ for $\phi=\chi$, \cite{Banas}, while Goebel and Sekowski considered the modulus of noncompact convexity associated with the Kuratowski measure $\alpha$, \cite {Goebel-Sekowski}. For $\phi =\beta$,  $\triangle _{X,\beta}(\varepsilon)$  represents the Dominguez-Lopez modulus of noncompact convexity.

\section{Introducing the new modulus}
\begin{definition}
Let $\phi$ be arbitrary measure of noncompactness on a complete metric space $X$. We define the function $\triangle'_{X,\phi}: [0,\phi (B_X)]\rightarrow [0,1]$ by
$$\triangle'_{X,\phi}(\varepsilon)=\inf \{1-d(0,co(A)) \ | \ A\subseteq \overline{B}_X, \ A \textrm{ $\phi$-minimal}, \ \phi(A)>\varepsilon \} \ .$$
\end{definition}

\noindent The modulus $\triangle'_{X,\phi}$ is a well defined function (see Theorem 1.2. \cite{Ayerbe_knjiga}). In the general case, because of the definition of infimum for an arbitrary measure of noncompactness $\phi$ on an arbitrary metric space $X$, we have that
\begin{equation}\label{odnosmodulaimodulaprim}
\triangle_{X,\phi}(\varepsilon)\leq \triangle'_{X,\phi}(\varepsilon) \ .
\end{equation}

\begin{theorem}\label{Tmodulmodulprimstrogaminimizabilnost}
Let $\phi$ be a strictly minimalizable measure of noncompactness on a metric space $(X,d)$. Then
$$\triangle'_{X,\phi}(\varepsilon)=\triangle_{X,\phi}(\varepsilon) \ .$$
\end{theorem}
\begin{proof}
Let $\phi$ be a strictly minimalizable measure of noncompactness and let $\eta >0$ be arbitrary. For arbitrary $\varepsilon \in [0,\phi(\overline{B}_X)]$, there exists $A\subseteq \overline{B}_X$, such that $\phi(A)>\varepsilon$ and
$$\triangle_{X,\phi}(\varepsilon)+\eta > 1-d(0,A).$$
Since $\phi$ is strictly minimalizable, there exists a $\phi$-minimal set $B\subset A$, such that $\phi(B)=\phi(A)$. Besides, we have  $co (B)\subseteq co (A)$, so that
$$1-d(0,co (B))\leq 1-d(0,co (A))<\triangle_{X,\phi}(\varepsilon)+\eta .$$
If we take the infimum over all the $\phi$-minimal sets $B$, such that $\phi(B)>\varepsilon$, we get that
$$\triangle'_{X,\phi}(\varepsilon)\leq \triangle_{X,\phi}(\varepsilon) +\eta  \ .$$
Since this  holds for arbitrary $\eta >0$, we have that
\begin{equation}\label{nejednakost1}
\triangle'_{X,\phi}(\varepsilon) \leq \triangle_{X,\phi}(\varepsilon) \ .
\end{equation}
Using (\ref{odnosmodulaimodulaprim}) and (\ref{nejednakost1}) we get the required equality.
\end{proof}

\begin{theorem}\label{Tmodulmodulprimminimizabilnost}
Let $\phi$ be a minimalizable measure of noncompactness on a metric space $(X,d)$. Then we have that
$$\triangle_{X,\phi}(\varepsilon)= \triangle'_{X,\phi}(\varepsilon) \ ,$$
for all $\varepsilon \in \left[0,\phi(\overline{B}_X)\right]$.
\end{theorem}
\begin{proof}
Let $\varepsilon \in \left[0,\phi(\overline{B}_X)\right]$ and $\eta >0$ be arbitrary. By the definition of the modulus $\triangle_{X,\phi}(\varepsilon)$, there exists $A^*\subseteq \overline{B}_X$, such that $\phi(A^*)>\varepsilon$ and
$$\triangle_{X,\phi}(\varepsilon)+\eta > 1-d(0,A^*) \ .$$
Let $\displaystyle \delta =\frac{\phi (A)-\varepsilon}{2}>0$. Since $\phi$ is minimalizable, there exists a $\phi$-minimal set $B\subset A^*$, such that
$$\phi(B)\geq \phi(A^*)-\delta=\frac{2\phi (A)+\varepsilon}{2}>\varepsilon$$
and
$$1-d(0,co (B))\leq \triangle_{X,\phi}(\varepsilon)+\eta \ .$$
If we take the infimum over all $\phi$-minimal sets $B$, such that $B \subseteq \overline{B}_X$ and $\phi(B)>\varepsilon$, we have that
$$\triangle'_{X,\phi}(\varepsilon)\leq \triangle_{X,\phi}(\varepsilon)+\eta \ .$$
Since $\eta$ is arbitrary, we conclude that
\begin{equation}\label{nejednakost2}
\triangle'_{X,\phi}(\varepsilon)\leq \triangle_{X,\phi}(\varepsilon) \ .
\end{equation}
Using (\ref{odnosmodulaimodulaprim}) and (\ref{nejednakost2}) we get the required equality.
\end{proof}

Let $X=l_p$ ($2\leq p <\infty$) be the space of $p$-summable sequences. Since every $\alpha$-minimal set is $\beta$-minimal (Lemma 2.9 \cite{Ayerbe_knjiga}), and measure $\beta$ is minimalizable measure of noncompactness on the space $l_p$, by the explicit expressions for $\triangle_{l_p,\beta}(\varepsilon)$ (Theorem 1.16.  \cite{Ayerbe_knjiga}) and  $\triangle_{l_p,\alpha}(\varepsilon)$ (Remark 1.17. \cite{Ayerbe_knjiga}), we have that
\begin{align*}
\triangle'_{l_p,\alpha}(\varepsilon)& = \inf \{1-d(0,co(A)) \ | \ A\subseteq \overline{B}_{l_p}, \ A \textrm{ $\alpha$-minimal}, \ \alpha(A)>\varepsilon \}\\
& \geq \inf \{1-d(0,co(A)) \ | \ A\subseteq \overline{B}_{l_p}, \ A \textrm{ $\beta$-minimal}, \ \beta(A)>\varepsilon \}\\
& = \triangle'_{l_p,\beta}(\varepsilon)=\triangle_{l_p,\beta}(\varepsilon)\\
& = 1-\sqrt[p]{1-\frac{\varepsilon^p}{2}}\\
& > 1-\sqrt[p]{1-\left(\frac{\varepsilon}{2}\right)^p}=\triangle_{l_p,\alpha}(\varepsilon) \ .
\end{align*}
This confirms the fact that in (\ref{odnosmodulaimodulaprim}) strict inequality can hold. This also justifies introducing the new modulus.

\begin{lemma}
Let $X$ be a separable metric space. Then we have that
$$\triangle_{X,\chi}(\varepsilon)=\triangle'_{X,\chi}(\varepsilon) \ .$$
\end{lemma}
\begin{proof}
Let $\varepsilon \in [0, \chi(B_X)]$ be arbitrary. By the definition of the modulus of noncompactness, for $\eta >0$ there exists $A\subset \overline{B}_X$, such that $\chi(A)>\varepsilon$ and
$$1-d(0, co(A))< \triangle_{X,\chi}(\varepsilon) +\eta \ .$$
Hence there exists a $\chi$-minimal set $B\subset A$, such that $\chi(A)=\chi(B)$ (\cite{benavides1986some}). For the set $B$, we have that the following inequalities hold
$$1-d(0, co(B))\leq 1-d(0, co(A))<\triangle_{X,\chi}(\varepsilon) +\eta \ .$$
If in the last relation we take the infimum over all the $\chi$-minimal sets $B$, such that $\chi(B)>\varepsilon$, we have that
$$\triangle'_{X,\chi}(\varepsilon)\leq \triangle_{X,\chi}(\varepsilon)+\eta \ .$$
Since $\eta>0$ is arbitrary, we get
$$\triangle'_{X,\chi}(\varepsilon)\leq \triangle_{X,\chi}(\varepsilon) \ .$$
If we apply (\ref{odnosmodulaimodulaprim}) for the measure $\chi$, then by the last inequality, we get the result.
\end{proof}

Relations between the moduli of noncompactness, associated with the standard measures of noncompactness $\alpha, \chi$ and $\beta$, and modulus of convexity are well known,
\begin{equation}\label{produzenanejednakostmodula}
\delta_X(\varepsilon)\leq \triangle_{X,\alpha}(\varepsilon) \leq \triangle_{X,\beta}(\varepsilon) \leq \triangle_{X,\chi}(\varepsilon) \ ,
\end{equation}
where $\delta_X$ is the modulus of convexity. Since the measure of noncompactness $\beta$ is minimalizable on every complete metric space, by the Theorem \ref{Tmodulmodulprimminimizabilnost} we have that $\triangle_{X,\beta}(\varepsilon)= \triangle'_{X,\beta}(\varepsilon)$. If we apply (\ref{odnosmodulaimodulaprim}) for the measure $\chi$, we get the following relations:
$$\triangle_{X,\alpha}(\varepsilon)\leq \triangle'_{X,\beta}(\varepsilon) \leq \triangle'_{X,\chi}(\varepsilon) \ .$$

\begin{lemma}{\label{Lnejednakostmodulaprimalfabeta}}
For the measures $\alpha $ and $\beta$ we have that
$$\triangle'_{X,\beta}(\varepsilon) \leq \triangle'_{X,\alpha}(\varepsilon). $$
\end{lemma}
\begin{proof}
If $A$ is  an $\alpha$-minimal set, then $A$ is also a $\beta$-minimal set and $\alpha (A)=\beta (A)$. Furthermore,
\begin{align*}
\triangle'_{X,\alpha}(\varepsilon) & = \inf\{ 1-d(0,co (A)): \ A\subseteq \overline{B}_{X}, \ A \textrm{ $\alpha$-minimal}, \ \alpha(A)>\varepsilon \} \\
 & \geq \inf \{ 1-d(0,co (A)): \ A\subseteq \overline{B}_{X}, \ A \textrm{ $\beta$-minimal}, \ \beta(A)>\varepsilon \} \\
 & =\triangle'_{X,\beta}(\varepsilon).
\end{align*}
\end{proof}

If we assume that $\alpha$ is a minimalizable measure of noncompactness on a space $X$, then by the Lemma \ref{Lnejednakostmodulaprimalfabeta} we can more precisely formulate the relation between the moduli for considered measures $\alpha, \beta$ and $\chi$. We have that
\begin{equation}\label{odnosmodulaprim2}
\triangle'_{X,\beta}(\varepsilon)= \triangle'_{X,\alpha}(\varepsilon) \leq \triangle'_{X,\chi}(\varepsilon) \ .
\end{equation}

\begin{lemma}\label{Lodnosmodulaprimalfahi}
For the measures $\alpha $ and $\chi$ we have that
\begin{equation}\label{odnosmodulaprim2}
\triangle'_{X,\alpha}(\varepsilon)\leq \triangle'_{X,\chi}(\varepsilon)
\end{equation}
\end{lemma}
\begin{proof}
Assume that there exists $\varepsilon_0 \in [0,\alpha(B_X))$, such that
\begin{equation*}
\triangle'_{X,\alpha}(\varepsilon_0)> \triangle'_{X,\chi}(\varepsilon_0) \ .
\end{equation*}
By the definition of the modulus $\triangle'_{X,\chi}(\varepsilon)$, there exists a $\chi$-minimal set $A^*\subset\overline{B}_X$, such that $\chi(A^*)>\varepsilon_0$ and
\begin{equation*}
\triangle'_{X,\alpha}(\varepsilon_0)> 1-d(0,co(A^*)) \ .
\end{equation*}
Because of the property of the infimum, for all $\alpha$-minimal sets $A\subset\overline{B}_X$, such that $\alpha(A)>\varepsilon_0$ we have that
\begin{equation}\label{vaznanejednakostuodnosmodulaprim2}
1-d(0,co(A))>1-d(0,co(A^*)) \ .
\end{equation}
If we assume that $A^*$ is $\alpha$-minimal, then since $\alpha(A^*)\geq \chi(A^*)>\varepsilon_0$ holds, inequality (\ref{vaznanejednakostuodnosmodulaprim2}) is true for the set $A^*$, which is impossible. Therefore, $A^*$ is not an $\alpha$-minimal set. This means that there exists an infinite set $B^*\subset A^*$, such that $\alpha(B^*)<\alpha(A^*)$. Since the set $B^*$ is bounded, there exists a set $C^*\subset B^*$ which is $\alpha$-minimal.\\
Assuming that $\alpha(C^*)>\varepsilon_0$ and applying (\ref{vaznanejednakostuodnosmodulaprim2}) for the set $C^*$, we get that
$$1-d(0,co(C^*))>1-d(0,co(A^*)) \ ,$$
which is not possible since $C^*\subset A^*$.\\
Assuming that $\alpha(C^*)\leq \varepsilon_0$ and since the set $C^*$ is a subset of the $\chi$-minimal set, we have that
$$\varepsilon_0<\chi(A^*)=\chi(C^*)\leq \alpha(C^*)\leq \varepsilon_0  \ .$$
Once again this is a contradiction. We conclude that the initial assumption is not sustainable.
\end{proof}

Using the Lemma \ref{Lnejednakostmodulaprimalfabeta} and Lemma \ref{Lodnosmodulaprimalfahi} we conclude that in the general case of the complete metric space $X$, we have that
$$\triangle'_{X,\beta}(\varepsilon) \leq \triangle'_{X,\alpha}(\varepsilon)\leq \triangle'_{X,\chi}(\varepsilon) \ .$$
Generally, using (\ref{odnosmodulaimodulaprim}) and (\ref{produzenanejednakostmodula}) we have that
\begin{equation}\label{generalnanejednakostmodulaprim}
\delta_X(\varepsilon)\leq \triangle'_{X,\beta}(\varepsilon) \leq \triangle'_{X,\alpha}(\varepsilon)\leq \triangle'_{X,\chi}(\varepsilon) \ .
\end{equation}

\begin{lemma}\label{L7}
Let $1\leq p<\infty$. If a bounded set is $\alpha$-minimal in the $l_p$ space, then it is $\chi$-minimal in the $l_p$ space.
\end{lemma}
\begin{proof}
Let $A\subseteq l_p$ be an $\alpha$-minimal set. This means that for all infinite subsets $B\subseteq A$, we have that
$\alpha(B)=\alpha(A)$. Using Corollary 4.5 \cite{Ayerbe_knjiga}, we conclude that $\alpha (A)=2^{\frac{1}{p}} \chi (A)$. Let $B\subseteq A$ be an arbitrary infinite set. Then $\alpha (B)=\alpha (A)$ and $\alpha (A)=2^{\frac{1}{p}}\chi (A)$. Since every set $B$  is a subset of the $\alpha$-minimal set, the set $B$ is $\alpha$-minimal itself.  So, $\alpha (B)=2^{\frac{1}{p}}\chi(B)$. Now we have that $2^{\frac{1}{p}}\chi(B)=2^{\frac{1}{p}}\chi(A)$, i.e. $\chi(B)=\chi(A)$. Since the set $B$ is arbitrary, we conclude that the set $A$ is $\chi$-minimal set.
\end{proof}

The equality $\alpha(A)=2^{\frac{1}{p}}\chi(A)$ holds for all $A\subset l_p$ in the $l_p$ spaces ($1\leq p<\infty$). So, if the set is $\chi$-minimal, it is also $\alpha$-minimal. Since for $1< p<\infty$ the spaces $l_p$ are reflexive, measure of noncompactness $\chi$ is minimalizable and equality $\triangle'_{l_p,\chi}(\varepsilon)=\triangle_{l_p,\chi}(\varepsilon)$ holds. Since the modulus $\triangle_{l_p,\chi}(\varepsilon)$ is a subhomegenous function (\cite{Banas3}), we have that the modulus $\triangle'_{l_p,\chi}(\varepsilon)$ is also a subhomegenous function. Using Lemma \ref{L7}, in the case of $l_p$ spaces, the following holds:
\begin{align*}
\triangle'_{l_p,\alpha}(\varepsilon)& = \inf \{1-d(0,co(A)) \ | \ A\subseteq \overline{B}_{l_p}, \ A \textrm{ $\alpha$-minimal}, \ \alpha(A)>\varepsilon \}\\
& = \inf \{1-d(0,co(A)) \ | \ A\subseteq \overline{B}_{l_p}, \ A \textrm{ $\chi$-minimal}, \ \chi(A)>2^{-\frac{1}{p}}\varepsilon \}\\
& = \triangle'_{l_p,\chi}(2^{-\frac{1}{p}}\varepsilon)\\
& \leq 2^{-\frac{1}{p}}\triangle'_{l_p,\chi}(\varepsilon) \\
& < \triangle'_{l_p,\chi}(\varepsilon) \ .
\end{align*}
This example shows that in the relation (\ref{odnosmodulaprim2}) we can have strict inequality.

\begin{theorem}
If the Kuratowski measure of noncompactness $\alpha$ is minimalizable on a complete metric space $X$, then $$\Delta_{X,\alpha}(\varepsilon)=\Delta_{X,\beta}(\varepsilon).$$
\end{theorem}
\begin{proof}
Since $\alpha$ is a minimalizable measure of noncompactness on a space $X$, using Theorem \ref{Tmodulmodulprimstrogaminimizabilnost} and Theorem \ref{Tmodulmodulprimminimizabilnost} we have that $\displaystyle \Delta_{X,\alpha}(\varepsilon)=\Delta_{X,\alpha}'(\varepsilon)$. Furthermore, since $\beta $ is a minimalizable measure of noncompactness on every complete metric space (Theorem 2.10. \cite{Ayerbe_knjiga}), we get that
$$\Delta_{X,\beta}'(\varepsilon)=\Delta_{X,\beta}(\varepsilon).$$
Based on the above and Lemma \ref{Lnejednakostmodulaprimalfabeta}, we conclude that
\begin{equation}\label{prva}
\Delta_{X,\alpha}(\varepsilon)\geq \Delta_{X,\beta}(\varepsilon).
\end{equation}
Using (\ref{produzenanejednakostmodula}) and (\ref{prva}) we have that
$$\Delta_{X,\alpha}(\varepsilon)=\Delta_{X,\beta}(\varepsilon) \ ,$$
which proves the claim.
\end{proof}

If we consider the assertion above in the contraposition, we can induce that if the moduli associated with the measures $\alpha$ and $\beta$ are different on some space, then the Kuratowski measure $\alpha$ is not minimalizable. So, $\alpha$ is not strictly minimalizable on the space either. Since
$$\Delta_{l_p,\alpha}(\varepsilon)=1-\left(1-\left(\frac{\varepsilon}{2}\right)^p\right)^\frac{1}{p}\neq 1-\left(1-\frac{\varepsilon^p}{2}\right)^\frac{1}{p}=\Delta_{l_p,\beta}(\varepsilon) \ ,$$
(Theorem 1.16. and Remark 1.17. , \cite{Ayerbe_knjiga}) for $p\geq 2$, we see that the measure of noncompactness $\alpha$ is not minimalizable on the space $l_p$.

\noindent It is known that for Day spaces $D_1$ and $D_\infty$ the following holds
$$\Delta_{D_1(D\infty),\alpha}(\varepsilon)=1-\left(1-\left(\frac{\varepsilon}{2}\right)^2\right)^\frac{1}{2}\neq 1-\left(1-\frac{\varepsilon^2}{2}\right)^\frac{1}{2}=\Delta_{D_1(D_\infty),\beta}(\varepsilon) \ ,$$
see \cite{Goebel-Sekowski}.
So, we get that the measure of noncompactness $\alpha$ is not minimalizable on these spaces.

Now we can define the characteristic for the modulus $ \triangle'_{X,\phi}$,
$$\varepsilon'_\phi(X)=\sup \{\varepsilon \ | \ \triangle'_{X,\phi}(\varepsilon)=0\} \ ,$$
analogously to how the characteristic of the modulus of noncompact convexity was defined. Because of (\ref{odnosmodulaimodulaprim}), it is clear that the following holds
$$\varepsilon'_\phi(X)\leq \varepsilon_\phi(X) \ .$$
Due to (\ref{generalnanejednakostmodulaprim}), we have the relation between these characteristics
$$\varepsilon'_\chi(X) \leq \varepsilon'_\alpha(X) \leq \varepsilon'_\beta(X) \leq \varepsilon_0(X) \ .$$

\section{Main result}
\begin{theorem}\label{Tneprekidnaodozdo}
Let $(X,d)$ be a metric space and let $\phi$ be a measure of noncompactness on $X$. Then the modulus $\triangle'_{X,\phi}(\varepsilon)$ is a continuous function from below on $[0,\phi (\overline{B}_X))$.
\end{theorem}

\begin{proof}
Let $\varepsilon _0\in [0,\phi (\overline{B}_X))$ be arbitrary and $\varepsilon <\varepsilon_0$. For arbitrary $\eta >0$ there exists $\phi$-minimal set $A\subset \overline{B}_X$, such that $\phi (A)>\varepsilon $ and
$$1-d(0, co (A))<\triangle'_{X,\phi}(\varepsilon)+\eta.$$
Let $\displaystyle k=1+\frac{1-d(0,A)}{2}$. Consider the set $A^*=kA\cap \overline{B}_X\subset \overline{B}_X$. $A^*$ is $\phi$-minimal set, as a subset of the $\phi$-minimal set $kA$. Besides, we have that
$$\phi (A^*)=\phi (kA)=k\phi (A)>k\varepsilon,$$ and
$$1-d(0,co (A^*))<1-d(0,co (A))<\triangle'_{X,\phi}(\varepsilon)+\eta.$$
Let $\displaystyle \delta =\varepsilon_0\left(1-\frac{1}{k}\right)$. Then for $\varepsilon \in (\varepsilon _0-\delta, \varepsilon_0)$
$$\phi (A^*)>k(\varepsilon _0-\delta)=\varepsilon _0,$$ holds and
$$\inf \{ 1-d(0,co(A^*)): A^*\subset \overline{B}_X,  \ \textrm{$A^*$ $\phi$-minimal }, \phi (A^*)>\varepsilon_0\}\leq \triangle'_{X,\phi}(\varepsilon)+\eta.$$
Since $\eta $ is arbitrary, we conclude that
\begin{equation}{\label{nejednakostdruganeprekidnost}}
\lim_{\varepsilon \rightarrow \varepsilon_0 -}\triangle'_{X,\phi}(\varepsilon)=\triangle'_{X,\phi}(\varepsilon_0).
\end{equation}
\end{proof}

\begin{theorem}\label{Tneprekidnaodozgo}
Let $(X,d)$ be a metric space and let $\phi$ be a minimalizable measure of noncompactness on $X$. The modulus $\triangle'_{X,\phi}(\varepsilon)$ is a continuous function from above on $[0,\phi (\overline{B}_X))$.
\end{theorem}

\begin{proof}
Let $\eta >0$ and $\varepsilon_0\in [0,\phi (\overline{B}_X))$ be arbitrary and fixed. Using the definition of the modulus $\triangle'_{X,\phi}(\varepsilon_0)$, we conclude that there exists a $\phi$-minimal set $A\subseteq \overline{B}_X$, such that $\phi (A)>\varepsilon_0$ and
$$1-d(0,co (A))<\triangle'_{X,\phi}(\varepsilon)+\eta \ .$$
Let $\xi >0$ be arbitrary. Since $\phi$ is a minimalizable measure of noncompactness, there exists a $\phi$-minimal set $B\subset A$, such that
$$\phi (B)\geq \phi (A)-\xi \ .$$
But, since the set $A$ is $\phi$-minimal, we have that $\phi (A)=\phi (B)$. Let $\delta =\phi (A)-\varepsilon _0$ and let $\varepsilon \in (\varepsilon_0,\varepsilon_0+\delta)$ be arbitrary. Due to
$$1-d(0,co (B))<\triangle'_{X,\phi}(\varepsilon)+\eta \ ,$$
and $\phi (B)>\varepsilon$, we have
$$\inf \{ 1-d(0,co (B)): B\subseteq \overline{B}_X, B \textrm{ $\phi$-minimal}, \ \phi(B)>\varepsilon \}\leq \triangle'_{X,\phi}(\varepsilon_0)+\eta \ ,$$
or equivalently,
$$\triangle'_{X,\phi}(\varepsilon) \leq \triangle'_{X,\phi}(\varepsilon_0)+\eta \ .$$ Since $\eta $ is arbitrary, we have that
\begin{equation}{\label{nejednakostprvaneprekidnost}}
\lim_{\varepsilon \rightarrow \varepsilon_0 +}\triangle'_{X,\phi}(\varepsilon)=\triangle'_{X,\phi}(\varepsilon_0),
\end{equation}
which proves that the modulus $\triangle'_{X,\phi}(\varepsilon)$ is a continuous function from above on $[0,\phi (\overline{B}_X))$.
\end{proof}

Using Theorem \ref{Tneprekidnaodozdo} and Theorem \ref{Tneprekidnaodozgo} we conclude that the new modulus $\triangle'_{X,\phi}(\varepsilon)$ associated with the minimalizable measure of noncompactness $\phi$ is a continuous function on $[0,\phi (\overline{B}_X))$. Hence, the modulus of noncompact convexity $\triangle_{X,\phi}(\varepsilon)$ is equal to the modulus $\triangle'_{X,\phi}(\varepsilon)$ for the minimalizable measure of noncompactness $\phi$ and we get the following assertion.

\begin{corollary}\label{posljedicaneprekidnost}
Let $(X,d)$ be a metric space. The modulus $\triangle_{X,\phi}(\varepsilon)$ associated with the minimali-zable measure of noncompactness $\phi$ is a continuous function on $[0,\phi (\overline{B}_X))$.
\end{corollary}

Clearly, Corollary \ref{posljedicaneprekidnost} also holds for the strictly minimalizable measure of noncompactness.

\end{document}